\documentclass[11pt, a4paper]{article}
\usepackage{amsfonts}
\usepackage{amsbsy}
\usepackage{amssymb}
\usepackage{amsmath}
\usepackage{amsmath,amscd}
\usepackage{amsthm}
\usepackage{fancyhdr}

\theoremstyle{plain}

\newtheorem{lemma}{Lemma}[section]

\newtheorem{proposition}{Proposition}[section]
\theoremstyle{definition}
\newtheorem{definition}{Definition}[section]
\newtheorem{remark}{Remark}[section]
\newtheorem{example}{Example}[section]

\begin{document}

\title{Generalized almost paracontact structures}
\author{Adara M. Blaga and Cristian Ida}
\medskip
\date{}

\maketitle
\begin{abstract}
The notion of generalized almost paracontact structure on the
generalized tangent bundle $TM\oplus T^*M$ is introduced and its
properties are investigated. The case when the manifold $M$ carries
an almost paracontact metric structure is also discussed. Conditions
for its transformed under a $\beta$- or a $B$-field transformation
to be also a generalized almost paracontact structure are given.
Finally, the normality of a generalized almost paracontact structure
is defined and a characterization of a normal generalized almost
paracontact structure induced by an almost paracontact one is given.
\end{abstract}

\medskip 
\begin{flushleft}
\strut \textbf{2010 Mathematics Subject Classification:} 53C15, 53D10, 53D18.

\textbf{Key Words:} paracontact structure; generalized geometry.   
\end{flushleft}

\section{Introduction}

Generalized complex geometry unifies complex and symplectic geometry
and proved to have applications in physics, for example, in quantum
field theory, providing new sigma models \cite{zu}. N. Hitchin
\cite{hi} initiated the study of generalized complex manifolds,
continued by M. Gualtieri whose PhD thesis \cite{gu} is an
outstanding paper on this subject. Afterwards, many authors
investigated the geometry of the generalized tangent bundle from
different points of view: M. Crainic \cite{cr} studied these
structures from the point of view of Poisson and Dirac geometry, H.
Bursztyn, G. R. Cavalcanti and M. Gualtieri \cite{bcg1}, \cite{bcg2}
presented a theory of reduction for generalized complex, generalized
K\"{a}hler and hyper-K\"{a}hler structures. Regarding also the
generalized K\"{a}hler manifolds, L. Ornea and R. Pantilie \cite{or}
discussed the integrability of the eigendistributions of the
operator $J_+J_-+J_-J_+$, where $J_{\pm}$ are the two almost
Hermitian structures of a bihermitian one. In \cite{orpa} they
introduced the notion of holomorphic map in the context of
generalized geometry. M. Abouzaid and M. Boyarchenko \cite{ab}
proved that every generalized complex manifold admits a canonical
Poisson structure. They also proved a local structure theorem and
showed that in a neighborhood, a "first-order approximation" to the
generalized complex structure is encoded in the data of a constant
$B$-field and a complex Lie algebra. A technical description of the
$B$-field was given by N. Hitchin \cite{hit} in terms of connections
on gerbes. Extending the almost contact structures to the
generalized tangent bundle, I. Vaisman \cite{va} introduced the
generalized almost contact structure and established conditions for
it to be normal. Y. S. Poon and A. Wade \cite{po} described the
particular cases coming from classical geometry, namely, when a
contact structure, an almost cosymplectic and an almost contact one
define a generalized almost contact structure. While the contact
structures are in correspondence with complex structures, the
paracontact structures are in correspondence with product
structures. Therefore, would be natural to consider paracontact
structures in the context of generalized geometry.

Our aim is to define on the generalized tangent bundle a
\textit{generalized paracontact structure} which naturally extends
the previous ones. By means of certain orthogonal symmetries of
$TM\oplus T^*M$, namely, the $\beta$- and $B$-transforms, in the
particular case when the generalized paracontact structure comes
from an almost paracontact one, we shall study its invariance under
$\beta$- and $B$-field transformations, respectively, and also
provide a necessary and sufficient condition for it to be normal
(Proposition \ref{p}).

Also in \cite{Sah} it is proved that such structures carry certain Lie bialgebroid/quasi-Lie algebroid structures.

\section{Definitions and properties}

The notion of almost paracontact structure was introduced by I.
Sato. According to his definition \cite{sa}, an \textit{almost
paracontact structure} $({\varphi},\xi,\eta)$ on an odd-dimensional
manifold $M$ consists of a $(1,1)$-tensor field $\varphi$, called
the \textit{structure endomorphism}, a vector field $\xi$, called
the \textit{characteristic vector field} and a $1$-form $\eta$,
called the \textit{contact form}, which satisfy the following
conditions:
\begin{enumerate}
  \item $\varphi^2=I-\eta\otimes \xi$;
  \item $\eta(\xi)=1$.
\end{enumerate}

Moreover, if $g$ is a pseudo-Riemannian metric on $M$ such that
$g({\varphi} X,{\varphi} Y)=-g(X,Y)+\eta(X)\eta(Y)$, for any $X$,
$Y\in \Gamma(TM)$, we shall call $({\varphi},\xi,\eta,g)$
\textit{almost paracontact metric structure}. Notice that from the
definition we deduce that ${\varphi} \xi=0$, $\eta\circ {\varphi}
=0$, $\eta=i_{\xi}g$, $g(\xi,\xi)=1$ and $g({\varphi}
X,Y)=-g(X,{\varphi} Y)$, for any $X$, $Y\in \Gamma(TM)$.

From the tangent bundle $TM$ we shall pass to the generalized
tangent bundle $TM\oplus T^*M$, whose sections are pairs of objects
consisting of a vector field and a $1$-form and we shall adopt the
notation $X+\alpha\in \Gamma(TM\oplus T^*M)$. Let
$g_0(X+\alpha,Y+\gamma):=\frac{\displaystyle 1}{\displaystyle
2}(\alpha(Y)+\gamma(X))$, $X+\alpha$, $Y+\gamma\in \Gamma(TM\oplus
T^*M)$, be the neutral metric on $TM\oplus T^*M$ (of signature
$(n,n)$, where $n$ is the dimension of $M$).

Extending this structure to the generalized tangent bundle $TM\oplus
T^*M$, we give the following definition:
\begin{definition}
We say that $(\Phi,\xi,\eta)$ is a generalized almost paracontact
structure if $\Phi$ is an endomorphism of the generalized tangent
bundle $TM\oplus T^*M$, $\xi$ is a vector field and $\eta$ is a
$1$-form on $M$ such that
\begin{enumerate}
  \item
$g_0(\Phi(X+\alpha),Y+\gamma)=-g_0(X+\alpha,\Phi(Y+\gamma))$, for
any $X+\alpha$, $Y+\gamma\in \Gamma(TM\oplus T^*M)$;
  \item $\Phi^2=\begin{pmatrix}
          I-\eta\otimes \xi & 0 \\
          0 & (I-\eta\otimes \xi)^* \\
        \end{pmatrix}$;
  \item $\Phi\begin{pmatrix}
          \eta\otimes \xi & 0 \\
          0 & (\eta\otimes \xi)^* \\
        \end{pmatrix}=0$;
\item $\parallel \xi+\eta\parallel_{g_0}=1$.
\end{enumerate}
\end{definition}

Taking into account the first relation in the definition, the
representation of the structure $\Phi$ by classical tensor fields is
$\Phi=\begin{pmatrix}
          \varphi & \beta \\
          B & -\varphi^* \\
        \end{pmatrix}$, where $\varphi$ is an endomorphism of the
        tangent bundle $TM$, $\varphi^*$ its dual map defined by
        $(\varphi^*\alpha)(X):=\alpha(\varphi X)$, $\alpha \in
        \Gamma(T^*M)$, $X\in \Gamma(TM)$, $\beta$ a bivector field
        and $B$ a $2$-form on $M$ (both of them skew-symmetric) and from the second relation
we obtain the following
        conditions:
        $$\left\{
            \begin{array}{ll}
              \varphi^2+\beta B=I-\eta\otimes \xi \\
              B\beta+ (\varphi^*)^2=(I-\eta\otimes \xi)^*\\
              \varphi \beta-\beta \varphi^*=0\\
B\varphi-\varphi^*B=0
            \end{array}
          \right.
        $$
which are equivalent to
$$\left\{
            \begin{array}{ll}
              \varphi^2=I-\eta\otimes \xi-\beta B \\
              \beta(\alpha,\varphi^* \gamma)=\beta(\varphi^* \alpha,\gamma)\\
              B(X,\varphi Y)=B(\varphi X,Y)
            \end{array}
          \right.,
        $$
for any $X+\alpha$, $Y+\gamma\in \Gamma(TM\oplus T^*M)$.

Finally, the last two relations imply $\beta(\eta, \cdot)=0$,
$B(\xi,\cdot)$=0, $\varphi \xi=0$, $\eta\circ \varphi=0$ and
respectively, $\eta(\xi)=1$. Remark that if $(\varphi,\xi,\eta)$ is an almost paracontact
structure, then $(\Phi,\xi,\eta)$ is a generalized almost
paracontact structure, where $\Phi:=\begin{pmatrix}
          \varphi & 0 \\
          0 & -\varphi^* \\
        \end{pmatrix}
$. Indeed, $\Phi^2:=\begin{pmatrix}
          \varphi^2 & 0 \\
          0 & (\varphi^*)^2 \\
        \end{pmatrix}
$ and for any $\alpha \in \Gamma(T^*M)$ and $X\in \Gamma(TM)$:
\begin{eqnarray*}
[(\varphi^*)^2\alpha](X)&:=&\varphi^*(\alpha(\varphi X)):=\alpha(\varphi^2X)=\alpha(X-\eta(X)\cdot \xi)\\
&=&\alpha(X)-\eta(X)\cdot \alpha(\xi)=[\alpha-\alpha(\xi)\cdot
\eta](X).
\end{eqnarray*}

We obtain
\begin{eqnarray*}
\Phi^2(X+\alpha)&:=&\varphi^2X+(\varphi^*)^2\alpha\\&=&X-\eta(X)\cdot
\xi+\alpha-\alpha(\xi)\cdot \eta\\&=&(X+\alpha)-[\eta(X)\cdot
\xi+\alpha(\xi)\cdot \eta]\\&=&I(X+\alpha)-F(X+\alpha),
\end{eqnarray*}
where $F(X+\alpha):=\eta(X)\cdot \xi+\alpha(\xi)\cdot \eta$. Then we
can write $F(X+\alpha)=JX+J^*\alpha$, for $JX:=\eta(X)\cdot
\xi=(\eta \otimes \xi)X$ and its dual map
$(J^*\alpha)(X):=\alpha(JX)=\alpha(\eta(X)\cdot
\xi)=\alpha(\xi)\eta(X)=[\alpha(\xi)\cdot \eta](X)$. Therefore,
$F=\begin{pmatrix}
                                                                     J & 0 \\
                                                                     0 & J^* \\
                                                                   \end{pmatrix}
$ and $$\Phi^2=I-F=\begin{pmatrix}
                                                                     I-J & 0 \\
                                                                     0 & (I-J)^* \\
                                                                   \end{pmatrix}=
                         \begin{pmatrix}
                           I-\eta\otimes \xi & 0 \\
                           0 & (I-\eta\otimes \xi)^* \\
                         \end{pmatrix}.
                                                                   $$

The other relations from the definition are obvious.

\section{On the generalized almost paracontact structure induced by an almost paracontact one}

In what follows we shall consider the case when the generalized
almost paracontact structure $(\Phi,\xi, \eta)$ comes from an almost
paracontact structure $(\varphi,\xi, \eta)$, namely,
$$\Phi:=\begin{pmatrix}
          \varphi & 0 \\
          0 & -\varphi^* \\
        \end{pmatrix}.
$$ In this case, we call $(\Phi,\xi, \eta)$ the generalized almost
paracontact structure induced by $(\varphi,\xi, \eta)$.

\begin{example}
Let $(\varphi_1,\xi_1,\eta_1)$ and $(\varphi_2,\xi_2,\eta_2)$ be two
almost paracontact structures on $M$ and for any $t\in
[0,\frac{\pi}{2}]$, consider the one-parameter family
$(\varphi_t,\xi_t,\eta_t)_{t\in [0,\frac{\pi}{2}]}$ defined by
$\varphi_t:=\cos t \cdot \varphi_1+\sin t \cdot \varphi_2$,
$\xi_t:=\cos t \cdot \xi_1+\sin t \cdot \xi_2$, $\eta_t:=\cos t
\cdot \eta_1+\sin t \cdot \eta_2$. Denote by
$\Phi_1:=\begin{pmatrix}
          \varphi_1 & 0 \\
          0 & -\varphi_1^* \\
        \end{pmatrix}
$ and $\Phi_2:=\begin{pmatrix}
          \varphi_2 & 0 \\
          0 & -\varphi_2^* \\
        \end{pmatrix}
$ the endomorphisms of the corresponding generalized almost
paracontact structures. If $\eta_i(\xi_j)=\delta_{ij}$,
$\varphi_i\xi_j=0$, $i,j\in \{1,2\}$ and
$\varphi_1\varphi_2+\varphi_2\varphi_1=-(\eta_1\otimes
\xi_2+\eta_2\otimes \xi_1)$, then $\Phi_t:=\cos t \cdot \Phi_1+\sin
t \cdot \Phi_2$, $t\in [0,\frac{\pi}{2}]$, defines a generalized
almost paracontact structure. Indeed, we get
$\varphi_1^*\varphi_2^*+\varphi_2^*\varphi_1^*=(\varphi_1\varphi_2+\varphi_2\varphi_1)^*$
and from our conditions we obtain $\Phi_t^2=\begin{pmatrix}
         A_t  & 0 \\
          0 & B_t \\
        \end{pmatrix}=\begin{pmatrix}
          I-\eta_t\otimes \xi_t & 0 \\
          0 & (I-\eta_t\otimes \xi_t)^* \\
        \end{pmatrix}$, where $A_t:=\cos^2 t\cdot \varphi_1^2+\sin^2 t\cdot\varphi_2^2+
          \cos t \cdot \sin t\cdot (\varphi_1\varphi_2+\varphi_2\varphi_1)$ and $B_t:=
          \cos^2 t\cdot (\varphi_1^*)^2+\sin^2 t\cdot(\varphi_2^*)^2+
          \cos t \cdot \sin t\cdot
          (\varphi_1^*\varphi_2^*+\varphi_2^*\varphi_1^*)$.
\end{example}

\subsection{Compatibility with generalized Riemannian metrics}

Let $(\varphi,\eta,\xi,g)$ be an almost paracontact metric structure
on $M$ and consider on $TM\oplus T^*M$ the generalized Riemannian
metric $\mathcal{G}_{\tilde{g}}$ induced by ${\tilde{g}}$, for
${\tilde{g}}$ a Riemannian metric compatible with $\varphi$
[${\tilde{g}}(\varphi X,Y)=-{\tilde{g}}(X,\varphi Y)$, for any $X$,
$Y\in \Gamma(TM)$]. A natural question is if the endomorphism of the
induced generalized almost paracontact structure $(\Phi,\eta,\xi)$
is compatible with this metric. First, recall that a
\textit{generalized Riemannian metric} $\mathcal{G}$ is a positive
definite metric on the generalized tangent bundle $TM\oplus T^*M$
such that
\begin{enumerate}
  \item
$g_0(\mathcal{G}(X+\alpha),\mathcal{G}(Y+\gamma))=g_0(X+\alpha,Y+\gamma)$,
for any $X+\alpha$, $Y+\gamma\in \Gamma(TM\oplus T^*M)$;
  \item $\mathcal{G}^2=I$.
\end{enumerate}

Representing $\mathcal{G}$ as $\mathcal{G}=\begin{pmatrix}
          \varphi & \sharp_{g_1} \\
          \flat_{g_2} & \varphi^* \\
        \end{pmatrix}
$, where $\varphi$ is an endomorphism of the tangent bundle $TM$,
$\varphi^*$ its dual map, $\flat_{g_i}(X):=i_Xg_i$, $X\in
\Gamma(TM)$ and $\sharp_{g_i}:=\flat_{g_i}^{-1}$, $i\in \{1,2\}$,
for $g_1$, $g_2$ Riemannian metrics on $M$, the two conditions are
equivalent to:
$$\left\{
    \begin{array}{ll}
      \varphi^2=I-\sharp_{g_1}\circ \flat_{g_2} \\
      g_i(X,\varphi Y)=-g_i(\varphi X, Y)
    \end{array}
  \right.,
$$
for any $X$, $Y\in \Gamma(TM)$, $i\in \{1,2\}$.

Let ${\tilde{g}}$ be a Riemannian metric on $M$ and consider the
positive definite generalized metric $\mathcal{G}_{\tilde{g}}$
\cite{ho}, which can be viewed as an automorphism of $TM\oplus
T^*M$, $\mathcal{G}_{\tilde{g}}:=\begin{pmatrix}
                0 & \sharp_{\tilde{g}} \\
                \flat_{\tilde{g}} & 0 \\
              \end{pmatrix}
$, where $\sharp_{\tilde{g}}$ is the inverse of the musical
isomorphism $\flat_{\tilde{g}}(X):=i_X{\tilde{g}}$, $X\in
\Gamma(TM)$.

\begin{proposition}
If $(\varphi,\eta,\xi,g)$ is an almost paracontact metric structure
on $M$ and ${\tilde{g}}$ is a Riemannian metric satisfying
${\tilde{g}}(\varphi X,Y)=-{\tilde{g}}(X,\varphi Y)$, for any $X$,
$Y\in \Gamma(TM)$, then the endomorphism $\Phi$ of the induced
generalized paracontact structure is compatible with the generalized
Riemannian metric $\mathcal{G}_{\tilde{g}}$, that is,
$\mathcal{G}_{\tilde{g}}\circ \Phi=-\Phi \circ
\mathcal{G}_{\tilde{g}}$.
\end{proposition}
\begin{proof}
For any $X+\alpha\in \Gamma(TM\oplus T^*M)$,
$\mathcal{G}_{\tilde{g}}(\Phi(X+\alpha)):=\mathcal{G}_{\tilde{g}}(\varphi
X-\varphi^*\alpha):=\sharp_{\tilde{g}}(\varphi^*\alpha)-\flat_{\tilde{g}}(\varphi
X)$. Therefore, for any $U\in \Gamma(TM)$,
${\tilde{g}}(\sharp_{\tilde{g}}(\varphi^*\alpha),U)=\alpha(\varphi
U)$ and $(\flat_{\tilde{g}}(\varphi X))(U)={\tilde{g}}(\varphi
X,U)$. But for any $X+\alpha\in \Gamma(TM\oplus T^*M)$,
$\mathcal{G}_{\tilde{g}}(X+\alpha):=\sharp_{\tilde{g}}(\alpha)+\flat_{\tilde{g}}(X)$
and so, for any $U\in \Gamma(TM)$,
${\tilde{g}}(\sharp_{\tilde{g}}(\alpha),U)=\alpha(U)$ and
$(\flat_{\tilde{g}}(X))(U)={\tilde{g}}(X,U)$. 

It follows
$${\tilde{g}}(\sharp_{\tilde{g}}(\varphi^*\alpha),U)=\alpha(\varphi
U)={\tilde{g}}(\sharp_{\tilde{g}}(\alpha),\varphi
U)=-{\tilde{g}}(\varphi(\sharp_{\tilde{g}}(\alpha)),U),$$ for any
$U\in \Gamma(TM)$ and so
$\sharp_{\tilde{g}}(\varphi^*\alpha)=-\varphi(\sharp_{\tilde{g}}(\alpha))$.

Also 
$$(\flat_{\tilde{g}}(\varphi X))(U)={\tilde{g}}(\varphi
X,U)=-{\tilde{g}}(X,\varphi U)=-(\flat_{\tilde{g}}(X))(\varphi
U)=-(\varphi^*(\flat_{\tilde{g}}(X)))(U),$$ for any $U\in \Gamma(TM)$
and so $\flat_{\tilde{g}}(\varphi
X)=-\varphi^*(\flat_{\tilde{g}}(X))$. 

Then, for any $X+\alpha\in
\Gamma(TM\oplus T^*M)$:
\begin{eqnarray*}
\mathcal{G}_{\tilde{g}}(\Phi(X+\alpha))&:=&\sharp_{\tilde{g}}(\varphi^*\alpha)-\flat_{\tilde{g}}(\varphi
X)\\&=&-(\varphi(\sharp_{\tilde{g}}(\alpha))-\varphi^*(\flat_{\tilde{g}}(X)))\\&:=&
-\Phi(\sharp_{\tilde{g}}(\alpha)+\flat_{\tilde{g}}(X))\\&=&-\Phi(\mathcal{G}_{\tilde{g}}(X+\alpha)).
\end{eqnarray*}
\end{proof}
\begin{remark}
From the previous computations we also deduce that
$\flat_{\tilde{g}}\circ \varphi=-\varphi^*\circ \flat_{\tilde{g}}$
(respectively, $\sharp_{\tilde{g}}\circ \varphi^*=-\varphi\circ
\sharp_{\tilde{g}}$).
\end{remark}

\subsection{Invariance under a $B$-field transformation}

Besides the diffeomorphisms, the Courant bracket (which extends the
Lie bracket to the generalized tangent bundle) admits some other
symmetries, namely, the \textit{$B$-field transformations}. Now we
are interested in what happens if we apply to the endomorphism
$\Phi$ a $B$-field transformation.

Let $B$ be a closed $2$-form on $M$ [viewed as a map
$B:\Gamma(TM)\rightarrow \Gamma(T^*M)$] and consider the
$B$-transform, $e^B:=\begin{pmatrix}
                       I & 0 \\
                       B & I \\
                     \end{pmatrix}
$. We can define $\Phi_B:=e^B\Phi e^{-B}$ which has the expression
$\Phi_B=\begin{pmatrix}
                       \varphi & 0 \\
                       B\varphi+\varphi^*B & -\varphi^* \\
                     \end{pmatrix}$ and for any $X+\alpha\in \Gamma(TM\oplus
T^*M)$, we have $$\Phi_B(X+\alpha)=\varphi X+B(\varphi X)+
                     \varphi^* (B(X))-\varphi^*\alpha.$$ For any $Y\in
                     \Gamma(TM)$, we get $$[B(\varphi X)+
                     \varphi^* (B(X))-\varphi^*\alpha](Y)=B(\varphi X,Y)+B(X,\varphi
                     Y)-(\varphi^*\alpha)(Y).$$
Note that if the $2$-form $B$ satisfies $B(\varphi X,Y)=-B(X,\varphi
Y)$, for any $X$, $Y\in \Gamma(TM)$, then $\Phi_B$ coincides with
$\Phi$. In particular, if $(\varphi,\eta,\xi,g)$ is an almost
para-cosymplectic metric structure and if we take $B(X,Y):=g(\varphi
X,Y)$, $X$, $Y\in \Gamma(TM)$, we obtain $$B(\varphi
X,Y):=g(\varphi^2X,Y)=-g(\varphi X,\varphi Y):=-B(X,\varphi Y)$$ and
$\Phi_B$ is just $\Phi$.

A sufficient condition on $B$ for $\Phi_B$ to define a generalized
almost paracontact structure is given by the following proposition:

\begin{proposition}
If the $2$-form $B$ satisfies $B(\varphi^2 X,Y)=B(\varphi X,\varphi
Y)$, for any $X$, $Y\in \Gamma(TM)$, then $(\Phi_B,\eta,\xi)$ is a
generalized almost paracontact structure.
\end{proposition}
\begin{proof}
Indeed, $\Phi_B^2=\begin{pmatrix}
                       \varphi^2 & 0 \\
                       B\varphi^2-(\varphi^*)^2B & (\varphi^*)^2 \\
                     \end{pmatrix}=\begin{pmatrix}
                           I-\eta\otimes \xi & 0 \\
                           0 & (I-\eta\otimes \xi)^* \\
                         \end{pmatrix}$.
\end{proof}

\begin{remark}
In the general case, if $\Phi$ is represented $\Phi=\begin{pmatrix}
          \varphi & \beta \\
          B & -\varphi^* \\
        \end{pmatrix}$, then its $B$-transform,
$\Phi_B=\begin{pmatrix}
                       \varphi-\beta B & \beta \\
                       B\varphi+\varphi^*B+B-B\beta B & -\varphi^*+B\beta \\
                     \end{pmatrix}$
defines a generalized almost paracontact structure.
\end{remark}

\subsection{Invariance under a $\beta$-field transformation}

Similarly we shall see what happens if we apply to the endomorphism
$\Phi$ a $\beta$-field transformation. Let $\beta$ be a bivector
field on $M$ [viewed as a map $\beta:\Gamma(T^*M)\rightarrow
\Gamma(TM)$] and consider the $\beta$-transform,
$e^{\beta}:=\begin{pmatrix}
                       I & \beta \\
                       0 & I \\
                     \end{pmatrix}
$. We can define $\Phi_{\beta}:=e^{\beta}\Phi e^{-{\beta}}$ which
has the expression $\Phi_{\beta}=\begin{pmatrix}
                       \varphi & -\varphi{\beta}-{\beta}\varphi^* \\
                       0 & -\varphi^* \\
                     \end{pmatrix}$ and for any $X+\alpha\in \Gamma(TM\oplus
T^*M)$, we have $$\Phi_{\beta}(X+\alpha)=\varphi
                     X-\varphi({\beta}(\alpha))-{\beta}(\varphi^*\alpha)-\varphi^*\alpha.$$
If the bivector field $\beta$ satisfies $\beta\circ
\varphi^*=-\varphi\circ \beta$, then $\Phi_{\beta}$ coincides with
$\Phi$.

A sufficient condition on $\beta$ for $\Phi_{\beta}$ to define a
generalized almost paracontact structure is given by the following
proposition:

\begin{proposition}
If the bivector field $\beta$ satisfies $\eta(\beta(\alpha))\cdot
\xi=\alpha(\xi)\cdot \beta(\eta)$, for any $\alpha\in \Gamma(T^*M)$,
then $(\Phi_{\beta},\eta,\xi)$ is a generalized almost paracontact
structure.
\end{proposition}
\begin{proof}
Indeed, $\Phi_{\beta}^2=\begin{pmatrix}
                       \varphi^2 & \beta(\varphi^*)^2-\varphi^2\beta \\
                       0 & (\varphi^*)^2 \\
                     \end{pmatrix}$
and for any $\alpha \in \Gamma(T^*M)$:
$$\beta((\varphi^*)^2\alpha)-\varphi^2(\beta(\alpha))=\beta(\alpha-\alpha(\xi)\cdot \eta)-
(\beta(\alpha)-\eta(\beta(\alpha))\cdot
\xi)=\eta(\beta(\alpha))\cdot \xi-\alpha(\xi)\cdot \beta(\eta)=0$$
and so $\Phi_{\beta}^2=\begin{pmatrix}
                           I-\eta\otimes \xi & 0 \\
                           0 & (I-\eta\otimes \xi)^* \\
                         \end{pmatrix}$.
\end{proof}

\begin{remark}
In the general case, if $\Phi$ is represented $\Phi=\begin{pmatrix}
          \varphi & \beta \\
          B & -\varphi^* \\
        \end{pmatrix}$, then its $\beta$-transform,
$\Phi_{\beta}=\begin{pmatrix}
                       \varphi+\beta B & -\varphi\beta-\beta\varphi^*+\beta-\beta B\varphi \\
                       B & -\varphi^*-B\beta \\
                     \end{pmatrix}$
defines a generalized almost paracontact structure.
\end{remark}

\subsection{Paracontactomorphisms}

We shall prove that a diffeomorphism between two almost paracontact
manifolds preserving the almost paracontact structure induces a
diffeomorphism between their generalized tangent bundles which
preserves the generalized almost paracontact structure.

Let $(M_1,\varphi_1,\xi_1,\eta_1)$ and
$(M_2,\varphi_2,\xi_2,\eta_2)$ be two almost paracontact manifolds.

\begin{definition}
We say that $f:(M_1,\varphi_1,\xi_1,\eta_1)\rightarrow
(M_2,\varphi_2,\xi_2,\eta_2)$ is a paracontactomorphism if $f$ is a
diffeomorphism and satisfies
$$\varphi_2\circ f_*=f_*\circ \varphi_1, \ \ f_*\xi_1=\xi_2.$$
\end{definition}

Remark that in this case, $f^*\eta_2=\eta_1$ is also implied.
Indeed, for any $X\in \Gamma(TM_1)$, applying $f_*$ to
$\varphi_1^2X=X-\eta_1(X)\cdot \xi_1$, we get
\begin{eqnarray*}f_*X-(f^*)^{-1}(\eta_1(X))\cdot f_*\xi_1&=&(f_*\circ \varphi_1)(\varphi_1 X)=\varphi_2((f_*\circ
\varphi_1)X)\\&=& \varphi_2^2(f_*X)=f_*X-\eta_2(f_*X)\cdot \xi_2
\end{eqnarray*}
and so, $\eta_2(f_*X)\circ f=\eta_1(X)$, for any $X\in
\Gamma(TM_1)$.

\begin{lemma}
If $f:(M_1,\varphi_1,\xi_1,\eta_1)\rightarrow
(M_2,\varphi_2,\xi_2,\eta_2)$ is a paracontactomorphism, then
$\varphi_1^*\circ f^*=f^*\circ \varphi_2^*$.
\end{lemma}
\begin{proof}
For any $X\in \Gamma(TM_1)$, $\alpha \in \Gamma(T^*M_2)$, $x\in
M_1$:
$$[[(\varphi_1^*\circ f^*)(\alpha)](X)](x)=[(f^*\alpha)(\varphi_1 X)](x)=\alpha_{f(x)}({f_*}_x((\varphi_1 X)_x))
=[\alpha(f_*(\varphi_1 X))](f(x))$$ and respectively,
\begin{eqnarray*}[[(f^*\circ \varphi_2^*)(\alpha)](X)](x)&=&(\varphi_2^*\alpha)_{f(x)}({f_*}_x(X_x))=[(\varphi_2^*\alpha)(f_*X)](f(x))
\\&:=&[\alpha(\varphi_2(f_*X))](f(x))=[\alpha(f_*(\varphi_1
X))](f(x)).
\end{eqnarray*}
\end{proof}

\begin{proposition}
Let $f:(M_1,\varphi_1,\xi_1,\eta_1)\rightarrow
(M_2,\varphi_2,\xi_2,\eta_2)$ be a paracontactomorphism. Then it
induces a diffeomorphism between their generalized tangent bundles,
$\tilde{f}(X+\alpha):=f_*X+(f^{-1})^*\alpha$, $X+\alpha\in
\Gamma(TM_1\oplus T^*M_1)$, such that $\Phi_2\circ
\tilde{f}=\tilde{f}\circ \Phi_1$ and $\tilde{f}(\xi_1+0)=\xi_2+0$.
\end{proposition}
\begin{proof}
Using the previous lemma, we obtain, for any $X+\alpha\in
\Gamma(TM_1\oplus T^*M_1)$:
\begin{eqnarray*}(\Phi_2\circ \tilde{f})(X+\alpha)&:=&\Phi_2(f_*X+(f^{-1})^*\alpha):=(\varphi_2\circ f_*)(X)+
(\varphi_2^*\circ (f^{-1})^*)(\alpha) \\&=&(f_*\circ
\varphi_1)(X)+((f^{-1})^*\circ
\varphi_1^*)(\alpha)\\&:=&\tilde{f}(\varphi_1
X+\varphi_1^*\alpha):=(\tilde{f}\circ \Phi_1)(X+\alpha).
\end{eqnarray*}

Also, $\tilde{f}(\xi_1+0)=f_*\xi_1+0=\xi_2+0$ and
$\tilde{f}(0+\eta_1)=0+(f^{-1})^*\eta_1=0+\eta_2$.
\end{proof}

\subsection{Normality of $(\Phi,\xi,\eta)$}

I. Vaisman \cite{va} defined normal generalized contact structures
and characterized them. We give an analogue definition for the
normality of a generalized almost paracontact structure like in the
generalized contact case:
\begin{definition}
A generalized almost paracontact structure is called normal if the
$M$-adapted generalized almost product structure on $M\times
\mathbb{R}$ is integrable.
\end{definition}

Precisely, in our particular case, if $(\Phi:=\begin{pmatrix}
                                                \varphi & 0 \\
                                                0 & -\varphi^* \\
                                              \end{pmatrix}
, \xi,\eta)$ is the generalized almost paracontact structure induced
by the almost paracontact one $(\varphi,\xi,\eta)$, then the
$M$-adapted generalized almost product structure is
$P=\begin{pmatrix}
                                                \varphi & \beta \\
                                                B & -\varphi^* \\
                                              \end{pmatrix}$, where
$\varphi^2=I-\beta B$, $\beta(\alpha,
\varphi^*\gamma)=\beta(\varphi^*\alpha,\gamma)$, $B(X,\varphi
Y)=B(\varphi X,Y)$, for any $X$, $Y\in \Gamma(TM)$ and $\alpha$,
$\gamma\in \Gamma(T^*M)$. Moreover, form the condition to be
$M$-adapted \cite{va} follows $\beta=\xi\wedge
\frac{\partial}{\partial t}$ and $B=\eta \wedge dt$. The
integrability of $P$ means that its Courant-Nijenhuis tensor field
\begin{eqnarray*}
\mathcal{N}_P(X+\alpha,Y+\gamma)&:=&[P(X+\alpha),P(Y+\gamma)]+P^2[X+\alpha,Y+\gamma]\\&&-
P[P(X+\alpha),Y+\gamma]-P[X+\alpha,P(Y+\gamma)],
\end{eqnarray*}
for $X+\alpha$,
$Y+\gamma\in \Gamma(TM\oplus T^*M)$, vanishes identically, where the
Courant bracket is given by
$$[X+\alpha,Y+\gamma]:=[X,Y]+L_X\gamma-L_Y\alpha+\frac{\displaystyle
1}{\displaystyle 2}d(\alpha(Y)-\gamma(X)).$$ Computing it we obtain the normality
condition for $(\Phi,\xi,\eta)$: $$\left\{
  \begin{array}{ll}
    N_{\varphi}(X,Y)-d\eta(X,Y)\cdot \xi=0 \\
    L_{\xi}\eta=0$, $L_{\xi}\varphi=0 \\
    (L_{\varphi X}\eta)Y-(L_{\varphi Y}\eta)X=0
  \end{array}
\right.,$$ for any $X$, $Y\in \Gamma(TM)$, where
$$N_{\varphi}(X,Y):=[\varphi X,\varphi Y]+\varphi^2[X,Y]-
\varphi[\varphi X,Y]-\varphi[X,\varphi Y].$$

\begin{proposition}\label{p}
The generalized almost paracontact structure $(\Phi:=\begin{pmatrix}
                                                \varphi & 0 \\
                                                0 & -\varphi^* \\
                                              \end{pmatrix}
, \xi,\eta)$ induced by the almost paracontact one
$(\varphi,\xi,\eta)$ is normal if and only if $(\varphi,\xi,\eta)$
is normal.
\end{proposition}
\begin{proof}
The first implication is trivial. For the converse one, it is known
that $(\varphi,\xi,\eta)$ is normal if and only if
$N_{\varphi}(X,Y)-d\eta(X,Y)\cdot \xi=0$. Moreover, in this case,
the relations $L_{\xi}\varphi=0$ and $(L_{\varphi
X}\eta)Y-(L_{\varphi Y}\eta)X=0$ are also implied. Indeed, taking
$Y:=\xi$ in the previous relation we obtain
$$[X,\xi]-\varphi[\varphi X,\xi]+\xi(\eta(X))\cdot \xi=0,$$
for any $X$, $Y\in \Gamma(TM)$ and for $X\mapsto \varphi X$, we get
$$0=[\varphi
X,\xi]-\varphi[\varphi^2 X,\xi]=-[\xi,\varphi
X]+\varphi[\xi,X]=-(L_{\xi}\varphi)X.$$

But, $(\varphi,\xi,\eta)$ is normal if the associated almost product
structures $E_1:=\varphi-\eta\otimes \xi$ and
$E_2:=\varphi+\eta\otimes \xi$ are integrable (that is, their
Nijenhuis tensor fields vanish identically) and $L_{\xi}\eta=0$.
Applying $\eta$ to $N_{E_1}(\varphi X,Y)=0$, we obtain
$$(L_{\varphi^2X}\eta)Y-(L_{\varphi Y}\eta)(\varphi X)=0,$$
for any $X$, $Y\in \Gamma(TM)$, which is equivalent to
$$(L_{X}\eta)Y-(L_{\eta(X)\cdot \xi}\eta)Y-(L_{\varphi Y}\eta)(\varphi X)=0,$$
for any $X$, $Y\in \Gamma(TM)$. For $X\mapsto \varphi X$, we get
$$(L_{\varphi X}\eta)Y-(L_{\varphi Y}\eta)X+(L_{\varphi Y}\eta)(\eta(X)\cdot \xi)=0,$$
for any $X$, $Y\in \Gamma(TM)$ and the last term is zero because
\begin{eqnarray*}
(L_{\varphi Y}\eta)(\eta(X)\cdot \xi)&:=&(\varphi Y)(\eta(X))-
\eta([\varphi Y,\eta(X)\cdot \xi])\\&=&(\varphi
Y)(\eta(X))+\eta(X)\eta([\xi,\varphi Y])-(\varphi
Y)(\eta(X))\\&=&\eta(X)[\xi(\eta(\varphi Y))-(\varphi
Y)(\eta(\xi))-(d\eta)(\xi,\varphi Y)]=0.
\end{eqnarray*}
\end{proof}

{\bf Acknowledgement.} The first author acknowledges the support by the
research grant PN-II-ID-PCE-2011-3-0921.

\small{

\medskip
\noindent
Adara M. Blaga\\
Department of Mathematics and Informatics\\
West University of Timi\c{s}oara\\
Bld. V. P\^{a}rvan, no. 4, 300223 Timi\c{s}oara, Rom\^{a}nia\\
adara@math.uvt.ro\\
\medskip

\noindent
Cristian Ida\\
Department of Mathematics and Computer Science\\
University Transilvania of Bra\c{s}ov\\
Address: Bra\c{s}ov 500091, Str. Iuliu Maniu 50, Rom\^{a}nia\\
email:\textit{cristian.ida@unitbv.ro}

\end{document}